\documentclass[11pt]{article}
\usepackage{amsmath, amssymb, amsthm, epsfig, verbatim, xcolor}

\newtheorem{theorem}{Theorem}[section]
\newtheorem{lemma}[theorem]{Lemma}
\newtheorem{proposition}[theorem]{Proposition}

\newtheorem{corollary}[theorem]{Corollary}

\newtheorem{definition}{Definition}[section]

\newcommand{\supp}{\operatorname{supp}}

\newcommand{\cupdot}{\mathbin{\mathaccent\cdot\cup}}

\begin{document}

\title{\textbf{On the number of minimal codewords in codes generated by the adjacency matrix of a graph}}
\author{Sascha Kurz\\
\small{sascha.kurz@uni-bayreuth.de}\\
\small{Mathematisches Institut, Universit\"at Bayreuth, Germany}}

\date{}

\maketitle

\begin{abstract}
  Minimal codewords have applications in decoding linear codes and in cryptography.  We study the number of minimal codewords in binary linear codes that arise by appending a unit 
  matrix to the adjacency matrix of a graph. 
\end{abstract}


\section{Introduction}
Given a linear code its minimal codewords are those whose supports, i.e., the set of nonzero coordinates, do not properly contain the support of another nonzero codeword. 
They have applications e.g.\ in secret sharing schemes \cite{carlet2005linear, Massey}, two-party computation \cite{CCP}, and decoding algorithms \cite{Agrell2,Hwang}.
Complete decoding is an NP-hard problem \cite{berlekamp1978inherent}, so that it is no surprise that determining the set of minimal codewords is also a hard problem. 
Their number can grow exponentially in the dimension or the length of the code. The cases where all codewords are minimal are called minimal codes (or intersecting codes 
in the binary situation). They have e.g.\ applications in combinatorics \cite{sloane1993covering}. Indeed, the set of minimal codewords is only known for a few classes 
of linear codes, including $q$-ary Hamming codes, see e.g. \cite{AB,Agrell}. For Reed--Muller codes the problem is only partially solved, see e.g.\ \cite{BM,SST} and 
the references cited therein.

Here we consider the concatenation of a unit matrix and the adjacency matrix of a graph as a generator matrix of a linear code and study the sets of minimal codewords. For 
some graph classes we can characterize the sets of minimal codewords and count them. We can fully solve the problem for complete multipartite graphs, paths, and cycles.  
We also state some lower and upper bounds for the number of minimal codewords in terms of graph parameters. For small numbers of vertices we determine the maximum 
and minimum number of minimal codewords of connected graphs. It turns out that the minimum number of minimal codewords is always attained by paths. In \cite{Minmin} graphs 
where associated to linear codes via their cycle space and the corresponding sets of minimal codewords are studied.

\section{Preliminaries}
An $[n,k]_q$ linear code $C$ is a $k$-dimensional subspace of $\mathbb{F}_q^n$.  Given a vector $x\in\mathbb{F}_q^n$, the support of $x$ is defined as 
$\supp(x)=\{i\,:\, x_i\neq 0, 1\leq i\leq n \}$.  A $k\times n$ matrix $G$ whose rows form a basis for $C$ is called a generator matrix. If $G=[I_k|A]$, where $I_k$ is the 
$k\times k$ identity matrix, then it is said to be systematic or in standard form. A nonzero codeword $c\in C$ is minimal if there does not exist a nonzero codeword $c'$ such 
that $\supp(c')\subsetneq \supp(c)$. Otherwise (including the case $c=\mathbf{0}$), we call the codeword $c$ non-minimal. Note that a codeword and its nonzero scalar multiples have the same 
support. We say that two codewords are equivalent if one is a scalar multiple of the other. We use the notation $M(C)$ for the number of non-equivalent minimal codewords of $C$. 
In the following we will mainly consider binary codes. For some known properties of minimal codewords we refer the interested reader e.g.\ to \cite{AB} and the references 
cited therein. 

Given a graph $\mathcal{G}=(V,E)$ with vertex set $V$ and edge set $E$, we denote by $C(\mathcal{G})$ the binary linear code that is generated by $[I_{\# V}|A]$, where 
$A$ is an adjacency matrix of $\mathcal{G}$. As a shorthand we use the notation $M(\mathcal{G})$ for $M(C(\mathcal{G}))$. Note that $M(\mathcal{G})$ does not depend 
on the labeling of the vertices. In the remaining part of the paper we will give brief definitions for most of the used notions from graph theory. For standard definitions 
like e.g.\ a connected graph or its connectivity components we refer the reader to some standard text book on graph theory like e.g.\ \cite{west2001introduction}.

First we observe that we can restrict ourselves to the study of $M(\mathcal{G})$ for connected graphs:
\begin{lemma}
  \label{lemma_connectivity_components}
  Let $\mathcal{G}=(V,E)$ consist of $r\ge 1$ connectivity components $\mathcal{G}_1=(V_1,E_1)$, $\mathcal{G}_r=(V_r,E_r)$, i.e., $\cupdot_{i=1}^r V_i=V$ and 
  $\cupdot_{i=1}^r E_i=E$. Then, $M(\mathcal{G})=\sum_{i=1}^r M(\mathcal{G}_i)$.  
\end{lemma}
\begin{proof}
  The statement is obvious from the direct sum $C(\mathcal{G})=\oplus_{i=1}^r C(\mathcal{G}_i)$.  
\end{proof}

\begin{lemma}
  \label{lemma_trivial_lower_bound}
  For every graph $\mathcal{G}=(V,E)$ we have $M(\mathcal{G})\ge \# V$.
\end{lemma}
\begin{proof}
  If $G=[I_{\# V}|A]$ is a generator matrix of $C(\mathcal{G})$, then it is easy to check that the $\# V$ rows of $G$ give minimal codewords.   
\end{proof}
This trivial lower bound is attained with equality for graphs without edges, i.e., $M((V,\emptyset))=\# V$. In order to obtain a more interesting problem, we 
define $m(n)$ as the minimum of $M(\mathcal{G})$, where $\mathcal{G}$ is a connected graph with $n$ vertices, i.e., we ask for the minimum number of 
minimal codewords a graph with $n$ vertices can give. Similarly, let $M(n)$ denote the maximum of $M(\mathcal{G})$, where $\mathcal{G}$ is a graph, not 
necessarily connected, with $n$ vertices.  

More generally, let $M_q(n,k)$ be the maximum and $m_q(n,k)$ the minimum of $M(C)$ for all $[n,k]_q$ codes $C$.  Bounds and some exact values on $M_q(n,k)$ and $m_q(n,k)$  
can be found in \cite{Agrell2,Maxmin2,Maxmin,AB,kiermaier2019minimum}. Obviously, we have
$$
  m_2(2n,n)\le m(n)\le M(n)\le M_2(2n,n).
$$  

Let $C$ be a linear $[k+t,k]_2$ code with systematic generator matrix $G$. By $g^i$ we denote the $i$th row of $G$, where $1\le i\le k$. For each subset $S\subseteq\{1,\dots,k\}$ let 
$c^S$ denote the sum of the rows of $G$ with indices in $S$, i.e., $c^S=\sum_{i\in S}g^i\in C$. For each codeword $c\in C$ let $c_S\in\mathbb{F}_2^k$ denote the systematic part of $c$, 
i.e., the restriction of $c$ to the first $k$ coordinates $c_1,\dots, c_k$. Similarly, for each codeword $c\in C$ let $c_I\in\mathbb{F}_2^t$ denote the information bits, i.e., 
the restriction of $c$ to the last $t$ coordinates $c_{k+1},\dots,c_{k+t}$. Next, we study some properties of minimal codewords in general binary linear codes. 

\begin{lemma}
  \label{lemma_zero_sum}
  Let $\emptyset\neq S\subseteq\{1,\dots,k\}$. If there exists a subset $\emptyset\neq T\subsetneq S$ with $c^T_I=\mathbf{0}$, then $c^S$ is non-minimal.
\end{lemma} 
\begin{proof}
  Since $\supp\!\left(c^{S\backslash T}_I\right)=\supp\!\left(c^{S}_I\right)$ and $\supp\!\left(c^{S\backslash T}_S\right)\subsetneq \supp\!\left(c^{S}_S\right)$, we have 
  $\supp\!\left(c^{S\backslash T}\right)\subsetneq \supp\!\left(c^{S}\right)$.
\end{proof}

\begin{lemma}
  \label{lemma_supp_restricted_to_I}
  Let $\emptyset\neq S\subseteq\{1,\dots,k\}$. The codeword $c^S$ is non-minimal iff there exists a subset $\emptyset\neq T\subsetneq S$ with $\supp(c^T_I)\subseteq \supp(c^S_I)$.
\end{lemma}
\begin{proof}
  Since $S\neq \emptyset$ we have $c^S\neq \mathbf{0}$. Thus, if $c^S$ is non-minimal, there exists a subset $\emptyset\neq T\subsetneq S$ with $\supp(c^T)\subsetneq \supp(c^S)$, 
  so that $\supp(c^T_I)\subseteq \supp(c^S_I)$. For the other direction let $\emptyset\neq T\subsetneq S$ with $\supp(c^T_I)\subseteq \supp(c^S_I)$. If $\supp(c^T_I)\neq \supp(c^S_I)$, 
  then $\supp(c^T_I)\subsetneq \supp(c^S)$ implies $\supp(c^T)\subsetneq \supp(c^S_I)$ so that $c^S$ is non-minimal by definition. If $\supp(c^T_I)= \supp(c^S_I)$, then 
  $c^{S\backslash T}_I=\mathbf{0}$ and we can apply Lemma~\ref{lemma_zero_sum}.
\end{proof}  

\begin{corollary}
  \label{cor_max_card_S}
  Let $c^S$ be a minimal codeword. Then, we have $1\le \#S\le t+1$. Moreover, if $\# S=t+1$, then $c^S_I=\mathbf{0}$.   
\end{corollary}
\begin{proof}
  The largest cardinality of a set of linearly independent vectors in $\mathbb{F}_2^t$ is $t$. Thus, if $\#S\ge t+1$, then there exists a subset $T\subseteq S$ with 
  $c^T_I=\mathbf{0}$ and $\#T\le t+1$. We finally apply Lemma~\ref{lemma_zero_sum} to conclude $\#S\le t+1$.  
\end{proof}

\begin{lemma}
  \label{lemma_zero_sum_characterization}
  Let $\emptyset\neq S\subseteq\{1,\dots,k\}$ be a subset such that $c^S_I=\mathbf{0}$. Then, $c^S$ is minimal iff $c^T_I\neq\mathbf{0}$ for all $\emptyset\neq T\subsetneq S$.
\end{lemma}
\begin{proof}
  Since $S\neq \emptyset$ we have $c^S\neq \mathbf{0}$. If $c^S$ is non-minimal, then there exists a subset $\emptyset\neq T\subsetneq S$ with $\supp(c^T)\subsetneq\supp(c^S)$. 
  Since $c^S_I=\mathbf{0}$ this implies $c^T_I=\mathbf{0}$. For the other direction we apply Lemma~\ref{lemma_zero_sum}. 
\end{proof}

We have already observed that $c^S$ is minimal for all subsets $S\subseteq\{1,\dots,n\}$ of cardinality $1$. In a code $C(\mathcal{G})$ obtained from a graph 
$\mathcal{G}$ also the case of cardinality $\#S=2$ can be characterized easily:

\begin{lemma}
  \label{lemma_common_neighbor}
  Let $\mathcal{G}=(V,E)$ be a graph and $C=C(\mathcal{G})$ be its associated code. For $S=\{v_1,v_2\}$ the codeword $c^S$ is minimal iff 
  $v_1$ and $v_2$ have a common neighbor $v_3$ (where we assume that the vertices $v_1$, $v_2$, and $v_3$ are pairwise different).
\end{lemma}
\begin{proof}
  If $v_1$ and $v_2$ do not have common neighbors, then $\supp\!\left(c^{\{v_1\}}\right)\subsetneq c^S$, so that $c^S$ is non-minimal. If $v_1$ and $v_2$ have a common 
  neighbor $v_3$ then $c^S_I$ has a one at position $v_3$ while $c^{\{v_1\}}$ and $c^{\{v_2\}}$ have a one at position $v_3$, so that $c^S$ is minimal. 
\end{proof}

A path between two vertices $u$ and $v$ is a sequence of distinct vertices $\left[v_0,\dots,v_l\right]$, such that $v_0=u$, $v_l=v$, and 
$\left\{v_i,v_{i+1}\right\}$ is an edge for all $0\le i<l$. Such a path is called a shortest path if $l$ is minimal. We also call $l$ the length of the 
path. In a connected graph the length of the shortest path between two vertices gives a metric, i.e., the distance between two vertices is the length of a shortest 
path connecting them. The diameter of a connected graph is the maximum distance between pairs of vertices. Graphs of diameter $1$ are called complete graphs and 
we will determine the corresponding number $M(\mathcal{G})$ of minimal codewords in Proposition~\ref{prop_complete_graph}. For graphs with diameter $2$ we have: 

\begin{corollary}
  For a graph $\mathcal{G}=(V,E)$ with diameter $2$ we have $M(\mathcal{G})\ge {{\# V+1}\choose 2} -\# E$.
\end{corollary}
\begin{proof}
  For each subset $S\subseteq V$ of cardinality $1$ the codeword $c^S$ is minimal, which gives $\# V$ minimal codewords. 
  Now consider the ${{\#V} \choose 2}$ subset $S=\{u,v\}$ of cardinality $2$. If $\{u,v\}$ is not an edge in $\mathcal{G}$, then $u$ and $v$ are 
  at distance $2$ in $\mathcal{G}$. In other words, $u$ and $v$ have a common neighbor, so that we can apply Lemma~\ref{lemma_common_neighbor} to deduce the 
  minimality of $c^S$. Since $\# V+{{\#V} \choose 2}={{\#V+1} \choose 2}$, we obtain the stated lower bound.
\end{proof}

\section{The value of $M(\mathcal{G})$ for some graph classes}

Given the number $n$ of vertices, we will always set $V=\{1,\dots,n\}$ in this subsection. 
By $K_n$ we denote the complete graph on $n$ vertices, i.e., $V=\{1,\dots,n\}$ and $E=\{\{i,j\}\,:\, 1\le i<j\le n\}$. Obviously, we have $M(K_1)=1$ and $M(K_2)=2$. 
\begin{proposition}
  \label{prop_complete_graph}
  For each integer $n\ge 3$ we have
  $$
    M(K_n)=2^{n-1}+{n\choose 2}.
  $$      
\end{proposition}
\begin{proof}
  For some subset $\emptyset\neq S\subseteq \{1,\dots,n\}$ we can easily describe $c^S_I$. If $\# S\equiv 0\pmod 2$, then $c^S_I$ equals $1$ at position $v$ 
  iff $v\in S$ and $0$ otherwise. In the other case, $\# S\equiv 1\pmod 2$, we have that $c^S_I$ equals $1$ at position $v$ iff $v\not\in S$ and $0$ otherwise. So, if 
  the cardinality of $S$ is even and at least four, then we can choose at subset $T\subsetneq S$ of cardinality $T$ with $\supp(C^T_I)\subsetneq \supp(c^S_i)$, 
  i.e., the codeword $c^S$ is non-minimal. If $\#S=2$, then we can apply $n\ge 3$ and Lemma~\ref{lemma_common_neighbor} to deduce that $c^S$ is minimal. This gives  
  ${n\choose 2}$ cases of minimal codewords.  
  
  If $\# S=1$, then $c^S$ is minimal, which amounts to $n={n\choose 1}$ cases. Now let $\#S$ be odd and at least $3$. We have $c_I^S=\mathbf{0}$ iff $S=\{1,\dots, n\}$. 
  The only proper subset $T$ with $c_I^T=\mathbf{0}$ is $T=\emptyset$. Now let $\emptyset\neq T\subsetneq S$. If $\# T\equiv 1\pmod 2$, then $\supp\left(c^T_I\right) 
  \not\subseteq \supp(c^S_I)$, since $T\backslash S\neq\emptyset$. If $\# T\equiv 0\pmod 2$, then $\supp\left(c^T_I\right) 
  \not\subseteq \supp(c^S_I)$, since $T\neq\emptyset$. Thus, $c^S$ is minimal and there are $\sum_{1\le i\le n\,:\, i\text{ odd}} {n\choose i}$ 
  cases in total.    
  
  Thus, we have $M(K_n)={n\choose 2}+\sum_{1\le i\le n\,:\, i\text{ odd}} {n\choose i}$, which can be simplified further. Since $\sum_{i=0}^n {n\choose i}=2^n$ and 
  $\sum_{i=0}^n (-1)^i{n\choose i}=0$ the sum of odd binomial coefficients $\sum_{1\le i\le n\,:\, i\text{ odd}} {n\choose i}$ equals $2^{n-1}$, so that we obtain the 
  proposed formula. 
\end{proof}
  
For two positive integers we denote by $K_{a,b}$ the complete bipartite graph with vertex classes of size $a$ and $b$, respectively, i.e., for 
$A=\{1,\dots, a\}$ and $B=\{a+1,\dots,a+b\}$ we define the graph via $V=A\cup B$ and $E=\left\{\{\alpha,\beta\}\,:\, \alpha\in A, \beta\in B\right\}$.  

\begin{proposition}
  \label{prop_complete_bipartite_graph}
  For positive integers $a,b$ we have
  $$
    M(K_{a,b})=a+b+{a\choose 2}+{b\choose 2}.
  $$  
\end{proposition}
\begin{proof}
  For some subsets $A'\subseteq \{1,\dots, a\}$ and $B'\subseteq \{a+1,\dots,a+b\}$ with $S:=A'\cup B'\neq \emptyset$ we can easily describe $c^S_I$. The 
  value of $c^S_I$ at a position $\alpha\in A$ equals $\# B' \,\operatorname{rem}\, 2$, i.e., the remainder of $\# B'$ divided by $2$. Similarly, the 
  value of $c^S_I$ at a position $\beta\in B$ equals $\# A' \,\operatorname{rem}\, 2$.
  
  Every non-zero codeword can be written as $c^S$ for some subset $\emptyset\neq S\subseteq \{1,\dots,a+b\}$. We decompose $S=A'\cup B'$, where $A'\subseteq A$ 
  and $B'\subseteq B$. If $\# S\ge 3$ and $\# A'\ge 2$, then let $\tilde{A}\subseteq A'$ with $\#\tilde{A}=\# A'-2$. With this, we have 
  $\supp\!\left(c_I^{\tilde{A}\cup B'}\right)\subseteq \supp(c^S_I)$, i.e., $c^S$ is not a minimal codeword. If $\#S\ge 3$ and $\#A'\le 1$, then $\#B'\ge 2$ and we 
  can choose $\tilde{B}\subseteq B'$ with $\#\tilde{B}=\#B'-2$. Since $\supp\!\left(c_I^{A'\cup \tilde{B}}\right)\subseteq \supp(c^S_I)$ we again conclude that 
  $c^S$ is not a minimal codeword. If $\# S=1$, then $c^S$ is a minimal codeword. If $\#S=2$, then we can apply Lemma~\ref{lemma_common_neighbor} and 
  conclude that $c^S$ is minimal iff either $S\subseteq A$ or $S\subseteq B$.       
\end{proof}

Note that $K_2=K_{1,1}$.

\begin{corollary}
  For the star graph $K_{1,n-1}$ we have $M(K_{1,n-1})=n+{{n-1}\choose 2}=\frac{n^2-n +2}{2}$ for all $n\ge 2$.
\end{corollary}

Next we want to consider graph arising if we join the centers of two stars by an edge:
\begin{proposition}
  \label{prop_two_stars}
  Let $\mathcal{G}$ be a graph with two vertices $u$ and $v$ that are joined by an edge. The other $a+b$ vertices, where $a,b\ge 1$, are vertices of degree $1$, 
  where $a$ of them have $u$ as their unique neighbor and the other $b$ of them have $v$ as their unique neighbor. With this, we have 
  $M(\mathcal{G})=2+a+b+{{a+1}\choose 2}+{{b+1}\choose 2}$.
\end{proposition}
\begin{proof}
  By $n=2+a+b$ we denote the number of vertices of the graph and by $C$ the code $C(\mathcal{G})$ associated to $\mathcal{G}$. For each subset $S\subseteq\{1,\dots,n\}$ 
  of cardinality $1$ the codeword $c^S$ is minimal. For subsets $S$ of cardinality $2$ we apply Lemma~\ref{lemma_common_neighbor}. Counting the number of pairs of vertices 
  with a common neighbor gives ${{a+1}\choose 2}+{{b+1}\choose 2}$ choices. It remains to show that $c^S$ is non-minimal if $\#S\ge 3$. First we note $c^T_I=c^S_I$ 
  if $S$ arises from $T$ by adding two of the $a$ neighbors of $u$ of degree $1$. So, $c^S$ is non-minimal in that case. By symmetry, the same is true for the 
  $b$ neighbors of $v$ of degree $1$. So, let $x$ be an arbitrary neighbor of $u$ of degree $1$ and $y$ be an arbitrary neighbor of $v$ of degree $1$. It suffices 
  to consider $S\subseteq\{x,u,v,y\}$. In the following table we consider all choices for $S$ and abbreviate $c^S_I$ by just four binary entries. The second and third 
  entry correspond to vertex $u$ and vertex $v$, respectively. The first entry corresponds to vertex $x$ or any other neighbor of $u$ of degree $1$, noting that those 
  entries are all equal. Similarly, the fourth entry corresponds to vertex $y$ or any other neighbor of $v$ of degree $1$.   
  \begin{center}
    \begin{tabular}{llll}
      \hline
      $S$       & $c^S_I$     & $S$     & $c^S_I$      \\ 
      \hline
      $\{u\}$   & $(1,0,1,0)$ & $\{v\}$ & $(0,1,0,1)$  \\
      $\{x\}$   & $(0,1,0,0)$ & $\{y\}$ & $(0,0,1,0)$  \\  
      \hline
      $\{u,v\}$ & $(1,1,1,1)$ & $\{x,y\}$ & $(0,1,1,0)$ \\ 
      $\{x,u\}$ & $(1,1,1,0)$ & $\{v,y\}$ & $(0,1,1,1)$ \\
      $\{x,v\}$ & $(0,0,0,1)$ & $\{u,y\}$ & $(1,0,0,0)$ \\      
      \hline
      $\{x,u,v\}$ & $(1,0,1,1)$ & $\{u,v,y\}$ & $(1,1,0,1)$ \\
      $\{x,u,y\}$ & $(1,1,0,0)$ & $\{x,v,y\}$ & $(0,0,1,1)$ \\
      \hline
      $\{x,u,v,y\}$ & $(1,0,0,1)$ & & \\ 
      \hline
    \end{tabular}  
  \end{center}   
  The proof is finished by the easy but a bit tedious task to check that for all $S\subseteq \{x,u,v,y\}$ with $\#S\ge 3$ there exists a subset 
  $\emptyset\neq T\subsetneq S$ with $\supp\!\left(c^T_I\right)\subseteq\supp\!\left(c^S_I\right)$, so that we can apply Lemma~\ref{lemma_supp_restricted_to_I} to 
  conclude that $c^S$ is non-minimal.
\end{proof}

For an integer $r\ge 1$ and positive integers $a_1,\dots,a_r$ we denote by $K_{a_1,\dots,a_r}$ the complete multipartite graph, i.e., the vertex set of the 
$n=\sum_{i=1}^r a_i$ vertices is partitioned into $r$ classes such that two vertices are connected by an edge iff the come from different classes.

\begin{proposition}
  \label{prop_complete_multipartite_graph}
  For each complete multipartite graph $\mathcal{G}=K_{a_1,\dots,a_r}$ with $r\ge 3$ we have
  $$
    M(\mathcal{G})=n+{n\choose 2} +\sum_{U\subseteq \{1,\dots,r\}\,:\, \#U\equiv 1\pmod 2,\# U\ge 3} \prod_{i\in U} a_i,
  $$
  where $n=\sum_{i=1}^r a_i$.
\end{proposition}
\begin{proof}
  Let us denote the $r$ vertex classes by $A_1,\dots, A_r$. Given a non-empty subset $S$ of the vertex set, we set $A_i'=A_i\cap S$ for all $1\le i\le r$. 
  Let $v\in A_j$ for some $1\le j\le r$. Then, the value of $c^S_I$ at position $v$ is given by $\#\left(S\backslash A_j\right)\,\operatorname{rem}\, 2$. 
  If $\# S=1$, then $c^S$ is minimal. If $\#S=2$, then we can use Lemma~\ref{lemma_common_neighbor} to deduce that $c^S$ is minimal. If $\#S\ge 3$ and 
  $\#\left(S\cap A_j\right)\ge 2$ for some index $1\le j\le r$, then we can choose a two-element subset $U\subseteq S\cap A_j$ and use 
  $c^{S\backslash U}_I\subseteq c^{S}_I$ to conclude that $c^S$ is non-minimal. It remains to consider the cases where $\# S\ge 3$ and $\#\left(S\cap A_i\right)\le 1$ 
  for all $1\le i\le r$. Similar as in the proof of Proposition~\ref{prop_complete_graph}, we easily conclude that $c^S$ is minimal iff $\#S\equiv 1\pmod 2$.
\end{proof}
 
We remark that $K_n=K_{1,\dots,1}$, where the complete multipartite graph has exactly $n$ vertex classes with cardinality $1$ each.

\medskip


For each integer $n\ge 2$ we denote by $P_n$ the graph whose edges are given by $E=\{\{i,i+1\}\,:\, 1\le i\le n-1\}$. The graph $P_n$ is also called a 
path of order $n$.
 
\begin{proposition}
  \label{prop_path}
  For each integer $n\ge 1$ we have
  $$
    M(P_n)=\left\lfloor\frac{(n+1)^2}{4}\right\rfloor=\left\lfloor\frac{n+1}{2}\right\rfloor\cdot\left\lceil\frac{n+1}{2}\right\rceil.
  $$ 
\end{proposition}
\begin{proof}
  Each non-zero codeword of $C\!\left(P_n\right)$ is given as $c^S$ for some subset $\emptyset\neq S\subseteq \{1,\dots, n\}$. For $\# S=1$ the codeword $c^S$ is minimal. Given $S$, 
  a maximal distance-$2$ chain $U$ is a subset of $S$ of the form $U=\{a,a+2,\dots,b-2,b\}$, where $a\equiv b\pmod 2$ and $a-2,b+2\notin S$. If 
  $U=\{a,a+2,\dots,b-2,b\}\neq\emptyset$ is a (maximal) distance-$2$ chain, then $\supp(c^U_I)=\{n+a-1,n+b+1\}\cap\{n+1,\dots,2n\}$. We have $\supp(c^U_I)=\emptyset$ 
  iff $a=1$ and $b=n$. For a suitable integer $r\ge 1$ let $U_1,\dots, U_r$ be the unique decomposition of $S$ into maximal distance-$2$ chains. We directly conclude that 
  the supports of $c^{U_i}$ and $c^{U_j}$ are disjoint for all $1\le i<j\le r$. Thus, $c^S$ cannot be minimal if $r\ge 2$ since 
  $\supp\!\left(c^{U_i}\right)\subseteq \supp\!\left(c^S\right)$ and $\supp\!\left(c^{U_i}\right)\neq \emptyset$ for $1\le i\le r$. Now suppose that $S$ itself is 
  a maximal distance-$2$ chain, i.e., there exist integers $a$ and $b$ with $S=\{a,a+2,\dots,b-2,b\}$. Each proper subset $\emptyset \neq T\subsetneq S$ has a decomposition 
  into $r\ge 2$ maximal distance-$2$ chains $U_1,\dots, U_r$. Note that $\# \supp\!\left(c^{U_i}_I\right)\ge 1$ for all $1\le i\le r$. So, if $\supp(c^T_I)\subseteq \supp(c^S_I)$, 
  then we have $r=2$, $\supp(c^T_I)= \supp(c^S_I)=\{a-1,b+1\}$, and $\{1,n\}\subseteq U_1\cup U_2$. From the formula for $\supp(c^{U_i}_I)$ we conclude 
  $T=U_1\cup U_2=\{1,3,\dots, a-2,b+2,b+4,\dots,n\}\not\subseteq S$ -- contradiction. Thus, $c^S$ is minimal. 
  
  \medskip
  
  Counting the maximal distance-$2$ chains gives
  \begin{eqnarray*}
    M(P_n)&=&\#\left\{ (a,b)\,:\,1\le a\le b\le n, a\equiv b\pmod 2\right\} \\ 
    &=& \sum_{i=1}^n \left\lceil \frac{n+1-i}{2}\right\rceil=\left\lfloor\frac{(n+1)^2}{4}\right\rfloor=\left\lfloor\frac{n+1}{2}\right\rfloor\cdot\left\lceil\frac{n+1}{2}\right\rceil.
  \end{eqnarray*}  
\end{proof}  
 
Note that the formula for $M(P_n)$ is also valid for the case $P_2=K_2=K_{1,1}$. 

For each integer $n\ge 3$ we denote by $C_n$ the graph whose edges are given by $E=\big\{\{i,i+1\}\,:\, 1\le i\le n-1\big\}\cup\big\{\{1,n\}\big\}$. The graph $C_n$ is also called a 
cycle of order $n$. With $\tau\colon\mathbb{Z}\to\{1,\dots,n\}$ being the unique mapping with $\tau(z)\equiv z\pmod{n}$ for all $z\in\mathbb{Z}$, we can also write the edge set of 
$C_n$ as $\big\{\{\tau(i),\tau(i+1)\}\,:\, 1\le i\le n\big\}$. The proof of Proposition~\ref{prop_path} can be adjusted slightly to determine $M(C_n)$. 

\begin{proposition}
  \label{prop_cycle}
  For each integer $n\ge 3$ we have
  $$
    M(C_n)=\left\{\begin{array}{rcl}\frac{n^2-2n+4}{2}&:&n\equiv 0\pmod 2,\\ n^2-n+2&:&\text{otherwise.}\end{array}\right.
  $$ 
\end{proposition}
\begin{proof}
  Each non-zero codeword of $C\!\left(C_n\right)$ is given as $c^S$ for some subset $\emptyset\neq S\subseteq \{1,\dots, n\}$. For $\# S=1$ the codeword $c^S$ is minimal. Given $S$, 
  a distance-$2$ chain $U$ is a subset of $S$ of the form $U=\{\tau(a),\tau(a+2),\dots,\tau(b-2),\tau(b)\}$, where $a\equiv b\pmod 2$ and $1\le a\le b\le 2n$. 
  We call $U$ maximal (in $S$) if either $\tau(a-2),\tau(b+2)\notin S$ or $\tau(a-2)=\tau(b)$. The meaning is that neither $\{\tau(a),\tau(a+2),\dots,\tau(b),\tau(b+2)\}$ 
  nor $\{\tau(a-2),\tau(a),\dots,\tau(b-2),\tau(b)\}$ is a proper superset of $U$ that is a subset of $S$, i.e., we cannot enlarge $U$ to a strictly larger distance-$2$ chain. 
  Given a distance-$2$ chain $U=\{\tau(a),\tau(a+2),\dots,\tau(b-2),\tau(b)\}$ we have $\supp(c^U_I)=\emptyset$ if $\tau(a-2)=b$ and $\supp(c^U_I)=\{n+\tau(a-1),n+\tau(b+1)\}$ 
  otherwise. For a suitable integer $r\ge 1$ let $U_1,\dots, U_r$ be the unique decomposition of $S$ into maximal distance-$2$ chains. We directly conclude that the supports 
  of $c^{U_i}$ and $c^{U_j}$ are disjoint for all $1\le i<j\le r$. Thus, $c^S$ cannot be minimal if $r\ge 2$. Next we show that $c^S$ is minimal iff $S$ is a (maximal) 
  distance-$2$ chain itself. Each subset $\emptyset\neq T\subsetneq S$ has a decomposition into $r\ge 2$ maximal distance-$2$ chains $U_1,\dots, U_r$, 
  where $\#\supp(c^{U_i}_I)=2$ for all $1\le i\le r$, which contradicts $\cup_{i=1}^r \supp(c^{U_i}_I)\subseteq \supp(c^S_I)$. Thus, it remains to count the number of 
  different maximal distance-$2$ chains $U=\{\tau(a),\tau(a+2),\dots,\tau(b-2),\tau(b)\}$. 
  
  If $n$ is even, then the case $\tau(a-2)=\tau(b)$ can occur exactly two times, i.e., for the cases $\{1,3,\dots,n-1\}$ and $\{2,4,\dots,n\}$. Otherwise, we can start at any 
  vertex $1\le a\le n$ and choose $b=a+2j$, where $0\le j\le (n-4)/2$ since $j=(n-2)/2$ would yield $\tau(a-2)=\tau(b)$. Thus, if $n$ is even, we have 
  $M(C_n)=2+n\cdot\tfrac{n-2}{2}+2=\tfrac{n^2-2n+4}{2}$. If $n$ is odd, then the case $\tau(a-2)=\tau(b)$ occurs iff $U=\{1,\dots,n\}$. Otherwise, we can start at any 
  vertex $1\le a\le n$ and choose $b=a+2j$, where $0\le j\le n-2$. Thus, if $n$ is odd, we have $M(C_n)=1+n\cdot(n-1)=n^2-n+2$  
\end{proof}
  
Note that the formula for $M(C_n)$ is also valid for the case $C_3=K_3$. 

In a bipartite graph $\mathcal{G}$ we may generalizes the idea of a distance-$2$ chain as follows. We can build up a new graph $\mathcal{G}'$ 
with the same vertex set as $\mathcal{G}$. Two vertices in $\mathcal{G}'$ are connected by an edge, by definition, if they are at distance exactly $2$ in $\mathcal{G}$. Similar 
as in the proof of Proposition~\ref{prop_path} one can show that for each minimal codeword $c^S$ the set $S$ induces a connected subgraph in $\mathcal{G}'$. However, $c^S$ can 
be non-minimal for the vertex set $S$ of a connected subgraph of $\mathcal{G}'$, i.e., we may only conclude an upper bound on $M(\mathcal{G})$. This e.g.\ happens in $K_{a,b}$ 
provided that $a$ and $b$ are large enough.     

Another variant of a distance-$2$ chain can lead to lower bounds.
\begin{definition}
  In a graph $\mathcal{G}$ an even path between two vertices $u$ and $v$ is a sequence of distinct vertices $\left[v_0,\dots,v_l\right]$ such that $v_0=u$, $v_l=v$, $l$ is an 
  even positive integer, and $\left\{v_i,v_{i+1}\right\}$ is an edge for all $0\le i\le l-1$. We call $\left[v_0,\dots,v_l\right]$ a shortest even path between 
  $u$ and $b$ if $l$ is minimal.
\end{definition}
Note that even in a connected graph there does not need to exist a shortest even path for two given vertices. Moreover, in the case of existence it does not need to be 
unique. 

\begin{lemma}
  \label{lemma_minimal_codeword_sufficient}
  Let $\mathcal{G}$ be a graph and $C$ be the associated binary linear code. For each shortest even path $\left[v_0,\dots,v_l\right]$ the codeword 
  $c^S$, where $S=\{0\le i\le l\,:\,i\equiv 0\pmod 2\}$, is minimal.
\end{lemma}
\begin{proof}
  As an abbreviation we set $E=\{0\le i\le l\,:\,i\equiv 0\pmod{2}\}$ and $O=\{0\le i\le l\,:\,i\equiv 1\pmod{2}\}$. Let $o\in O$ and $e\in E$ such that 
  $\left\{v_o,v_e\right\}$ is an edge in $\mathcal{G}$. Since $l$ is minimal by assumption we have $e\in\{o-1,o+1\}$. Thus, for each $o\in O$ the vertex 
  $v_o$ has exactly two neighbors $v_e$ where $e\in E$. So, in $c^S_I$ the entries at the positions $v_o$ for $o\in O$ are equal to zero. Now let 
  $\emptyset\neq T\subsetneq S=E$ be an arbitrary subset, $t_{min}$ be the minimal element in $T$, and $t_{max}$ be the maximal element in $T$. If $t_{min}>0$, then 
  the entry at position $v_{t_{min}-1}$ in $c^T_I$ is equal to one so that $c^T\not\subseteq c^S$. Similarly, if $t_{max}<l$, then the entry at position $v_{t_{max}+1}$ 
  in $c^T_I$ is equal to one so that again $c^T\not\subseteq c^S$. Since $t_{min}=0$ and $t_{max}=l$ imply $T=S$ the codeword $c^S$ is indeed minimal.   
\end{proof}
Note that Lemma~\ref{lemma_common_neighbor} says that for a two element subset $S\subseteq\{1,\dots,n\}$ the codeword $c^S$ is exactly minimal if there 
exists a (shortest) even path of length $2$ between the elements of $S$.

\begin{lemma}
  \label{lemma_different_codewords_shortest_even_path}
  Let  $\left[v_0,\dots,v_l\right]$ and $\left[v_0',\dots,v_{l'}'\right]$ be two shortest even paths. If $$\left\{v_i\,:\,0\le i\le l, i\equiv 0\pmod 2\right\}=
  \left\{v_i'\,:\,0\le i\le l',i\equiv 0\pmod 2\right\},$$ then $\left\{v_0,v_l\right\}= \left\{v_0',v_{l'}'\right\}$.  
\end{lemma}
\begin{proof}
  First we note that $l=l'$. Now choose indices $0\le i,j\le l$ such that $v_0'=v_i$ and $v_l'=v_j$ with $i\equiv j\equiv 0\pmod 2$. Depending on whether $i<j$ or $i>j$,   
  either $\left[v_i,v_{i+1},\dots,v_j\right]$ or $\left[v_i,v_{i-1},\dots,v_j\right]$ is an even path from $v_0'$ to $v_l'$. Since we assume that $l$ is the shortest 
  possible length of such a path, we have $\{i,j\}=\{0,l\}$, i.e., $\left\{v_0,v_l\right\}= \left\{v_0',v_{l'}'\right\}$. 
\end{proof}

We remark that in an odd cycle, i.e., $C_n$ where $n$ is odd, the shortest even path between any two neighbored vertices uses all $n$ vertices.

\begin{lemma}
  \label{lemma_lower_bound_shortest_even_path} 
  Let $\mathcal{G}$ be a connected graph and $\mathcal{T}$ be a spanning tree of $\mathcal{G}$. Considering $\mathcal{T}$ as a bipartite graph, we denote the number 
  of vertices in the two color classes by $a$ and $b$. With this, we have $M(\mathcal{G})\ge a+b+{a\choose 2}+{b\choose 2}$.
\end{lemma}
\begin{proof}
  Let $C$ be the code $C(\mathcal{G})$ induced by the graph $\mathcal{G}$. For each of the $n$ subsets $S\subseteq\{1,\dots,n\}$ of cardinality $1$ the codeword 
  $c^S$ is minimal. These are $a+b$ minimal codewords. Now consider two vertices $u$ and $v$ of the same color class in $\mathcal{T}$. Due to the construction 
  of the coloring, there exists an even path between $u$ and $v$ in $\mathcal{T}$, which is also an even path between $u$ and $v$ in $\mathcal{G}$. If the path 
  is not already a shortest even path, then pick one. So, for every two vertices $u$ and $b$ of the same color class (in $\mathcal{T}$) there exists a shortest 
  even path $\left[v_0,\dots,v_l\right]$ between $u$ and $v$ in $\mathcal{G}$, so that Lemma~\ref{lemma_minimal_codeword_sufficient} implies that $c^S$ is 
  minimal, where $S=\left\{v_i\,:\,0\le i\le l, i\equiv 0\pmod 2\right\}$. There are ${a\choose 2}+{b\choose 2}$ choices and by 
  Lemma~\ref{lemma_different_codewords_shortest_even_path} all of them lead to different minimal codewords $c^S$, where $\# S\ge 2$.      
\end{proof}

We remark that the lower bound of Lemma~\ref{lemma_lower_bound_shortest_even_path} is attained with equality in Proposition~\ref{prop_complete_bipartite_graph}, 
i.e., for complete bipartite graphs, and in Proposition~\ref{prop_two_stars}. In Theorem~\ref{thm_minimum} we will use Lemma~\ref{lemma_lower_bound_shortest_even_path} 
to determine a formula for the minimum number $m(n)$ of minimal codewords of a connected graph with $n$ vertices. 

An induced subgraph of a graph $\mathcal{G}=(V,E)$ is a graph whose vertex set is a subset $S\subseteq V$ and whose edges are given by the elements of $E$  
where both vertices are contained in $S$. If $\mathcal{G}'=(V',E')$ is an induced subgraph of $\mathcal{G}=(V,E)$ and $c^S$ a minimal codeword in $\mathcal{G}'$, 
where $S\subseteq V'$, then $c^S$ is also a minimal codeword in $\mathcal{G}$. An odd cycle is an induced subgraph that is isomorphic to $C_l$, where $l$ is odd. 
\begin{proposition}
  \label{prop_induced_odd_cycle}
  Let $\mathcal{G}=(V,E)$ be a connected graph. If $S\subseteq V$ induces an odd cycle in $\mathcal{C}$, then $c^S$ is a minimal codeword in $C(\mathcal{G})$.   
\end{proposition}
\begin{proof}
  Let $\mathcal{G}'=(S,E')$ be the subgraph induced by $S$ and $C=C(\mathcal{G}')$ the binary code associated with $\mathcal{G}'$. In $C$ we have $c^S_I=\mathbf{0}$. 
  We can easily check that $c^T_I\neq\mathbf{0}$ for all $\emptyset\neq T\subsetneq S$, so that Lemma~\ref{lemma_zero_sum_characterization} gives that $c^S$ is minimal 
  in $C$. As noted above, $c^S$ is also minimal in $C(\mathcal{G})$.  
\end{proof}

Another lower bound, using more common graph invariants, for $M(\mathcal{G})$ is:
\begin{lemma}
  \label{lower_bound_graph_parameters_1}
  Let $\mathcal{G}$ be a graph with $n$ vertices, maximum degree $\Delta$, and $t$ triangles. Then, we have $M(\mathcal{G})\ge n +{\Delta\choose 2}+t$.
\end{lemma}
\begin{proof}
  For each of the $n$ subsets $S\subseteq\{1,\dots,n\}$ of cardinality $1$ the codeword $c^S$ is minimal. If $v$ is a vertex of degree $\Delta$, then for any pair $S$ 
  of two neighbors of $v$ we can apply Lemma~\ref{lemma_common_neighbor} to deduce that $c^S$ is minimal. Note that there are ${\Delta\choose 2}$ 
  choices. If $S$ consists of the three vertices of a triangle, then $c^S$ is minimal, see Proposition~\ref{prop_induced_odd_cycle}. 
\end{proof}

Next we want to study the special situation where all non-zero codewords are minimal.
\begin{proposition}
  If $\mathcal{G}$ is a graph with $n\ge 1$ nodes and $M(\mathcal{G})=2^n-1$, then $\mathcal{G}=K_1$ or $\mathcal{G}=K_3$.
\end{proposition}
\begin{proof}
  Due to \cite[Theorem 2(iii)]{sloane1993covering} an $[N,k]_2$ code whose non-zero codewords are all minimal satisfies $N\ge 3(k-1)$. 
  In our situation we have $N=2n$ and $k=n$, so that $n\le 3$. If $\mathcal{G}$ contains an isolated vertex, then $M(\mathcal{G})\le 1+M(n-1)\le 
  1+2^{n-1}-1$, which is strictly less than $2^n-1$ for $n\ge 2$. Thus, it suffices to consider the connected graphs with up to $3$ vertices: 
  $M(P_1)=1$, $M(P_2)=2$, $M(P_3)=4$, and $M(K_3)=7$, see Proposition~\ref{prop_path} and Proposition~\ref{prop_complete_graph}.  
\end{proof} 
 
\section{Exact values for small parameters}

The aim of this subsection is to determine the exact value of $M(n)$ and $m(n)$ for $1\le n\le 10$. Given Lemma~\ref{lemma_connectivity_components} 
it suffices to consider connected graphs. We note that there are already 11\,716\,571 non-isomorphic connected graphs, which we have enumerated using 
the software package \texttt{geng} \cite{mckay1981practical}. For each connected graph $\mathcal{G}$ we determine $M(\mathcal{G})$ by exhaustive enumeration. 

\begin{center}
  \begin{tabular}{|c|c|c|c|c|c|c|c|c|c|c|}
    \hline 
    $n$     & 1 & 2 & 3 &  4 &  5 &  6 &  7 &   8 &   9 &  10 \\ 
    \hline
    $m(n)$  & 1 & 2 & 4 &  6 &  9 & 12 & 16 &  20 &  25 &  30 \\
    \hline 
    $M(n)$  & 1 & 2 & 7 & 14 & 26 & 47 & 99 & 190 & 355 & 682 \\
    \hline
  \end{tabular}    
\end{center}
The maximum $M(n)$ is attained for $3\le n\le 6$ by a 
complete graph $K_n$, while the cases $n\in\{7,8,9,10\}$ need other constructions. For $n=10$ there are $22$ isomorphism types of graphs that attain 
the maximum of $682$ minimal codewords. Those graphs are quite diverse, i.e., their number of edges lies between $21$ and $32$, the minimum degree is either $4$ or 
$5$, and the maximum degree ranges from $5$ to $9$. For $1\le n\le 10$, the minimum value $m(n)$ is attained by a path $P_n$. This observation is also true in general. 

\begin{theorem}
  \label{thm_minimum}
  For each integer $n\ge 1$ we have
  $$
    m(n)=\left\lfloor\frac{(n+1)^2}{4}\right\rfloor=\left\lfloor\frac{n+1}{2}\right\rfloor\cdot\left\lceil\frac{n+1}{2}\right\rceil.
  $$
\end{theorem}
\begin{proof}
  Due to Proposition~\ref{prop_path} it remains to show the corresponding lower bound. To this end we apply Lemma~\ref{lemma_lower_bound_shortest_even_path}. 
  Setting $b=n-a$ we obtain the lower bound $n+{a\choose 2}+{{n-a}\choose 2}=a^2-na+\frac{n(n+1)}{2}$, which is a quadratic polynomial in $a$. Over the real 
  numbers the minimum is attained for $a=n/2$, which gives $m(n)\ge \frac{n(n+2)}{4}$. If $n$ is even this matches the statement. If $n$ is odd we can upround 
  $\frac{n(n+2)}{4}$ to $\frac{(n+1)^2}{4}$ since $m(n)$ is an integer.   
\end{proof}
We remark that all connected graphs $\mathcal{G}$ with $n$ vertices and $M(\mathcal{G})=m(n)$ are bipartite, since Proposition~\ref{prop_induced_odd_cycle} would give 
an additional minimal codeword that is not counted in Lemma~\ref{lemma_lower_bound_shortest_even_path}. If $\mathcal{T}$ is a tree such that all vertices with degree strictly 
larger than $1$ are contained on a path, then it can be easily shown that the lower bound of Lemma~\ref{lemma_lower_bound_shortest_even_path} is attained with 
equality. If the cardinalities of the two color classes of the bipartite tree $\mathcal{T}$ differ by at most $1$, then we have $M(\mathcal{T})=m(n)$, where $\mathcal{T}$ 
consists of $n$ vertices. We remark that one can also construct connected graphs $\mathcal{G}$ that contain $4$-cycles and satisfy $M(\mathcal{G})=m(n)$ for their number 
$n$ of vertices. The determination of $M(n)$ and the description of an infinite family of graphs attaining $M(\mathcal{G})=M(n)$ is an interesting open problem.

\section*{Acknowledgments} 
The author wishes to thank Romar dela Cruz for introducing him to the problem of minimal codewords and suggesting the construction of a linear code from the 
adjacency matrix of a graph.


\end{document}